\documentclass{article}
\usepackage [utf8] {inputenc}
\usepackage{mathtools}
\usepackage{amsthm}
\usepackage{amssymb}
\usepackage{enumerate}

\newtheorem{theorem}{\textit{Theorem}}
\newtheorem{lemma}{\textit{Lemma}}
\newtheorem*{corollary}{\textit{Corollary}}

\title{Subgroups of minimal index in polynomial time}
\author{S.V.~Skresanov\thanks{\rm The work is supported by Russian Science Foundation (project 14-21-00065)}}
\date{\vspace{-5ex}}
\begin{document}

\vspace{-4ex}
{\hfill \small MSC 20E28, 20B40, 05E18}

{\let\newpage\relax\maketitle}

\begin{abstract}
	Let \( G \) be a finite group and let \( H \) be a proper subgroup of \( G \) of minimal index.
	By applying an old result of Y.~Berkovich, we provide a polynomial algorithm for computing \( |G : H| \) 
	for a permutation group~\( G \). Moreover, we find \( H \) explicitly if \( G \) is given by a Cayley table.
	As a corollary, we get an algorithm for testing whether a finite permutation group acts on a tree or not.
\end{abstract}
\noindent{\small\textbf{Keywords:} subgroup of minimal index, minimal permutation representation, 
group representability problem, group representability on trees, permutation group algorithms.}

\section{Introduction}

In \cite{dutta} S.~Dutta and P.P.~Kurur introduced the following:
\medskip

\textbf{Group representability problem. }\textit{Given a group \( G \)
and a graph \( \Gamma \) decide whether there exists a nontrivial
homomorphism from \( G \) to the automorphism group of \( \Gamma \).}
\medskip

By \cite[Theorem~3]{dutta}, the graph isomorphism problem
reduces to the abelian group representability problem, so the latter inherits
the notorious difficulty of the former. 

As an attack from a different angle,
one can consider the problem of group representability on trees.
In \cite{dutta} authors speculate that there might be no polynomial
algorithm even for such a restriction. Nevertheless, in \cite[Theorems~6 and 8]{dutta} 
they provide a polynomial reduction of that problem to the 
\medskip

\textbf{Permutation representability problem. }\textit{Given a group \( G \) and a positive integer \( n \), 
decide whether there exists a nontrivial homomorphism from \( G \)
into the symmetric group \( Sym_n \).}
\medskip

Denote by \( \kappa(G) \) the degree of a minimal (not necessarily faithful)
nontrivial permutation representation of \( G \). Since such permutation representations
are always transitive, we see that \( \kappa(G) = \min \{ |G : H| \mid H < G \} \).
Notice that permutation representability problem reduces to the task of computing \( \kappa(G) \),
since for \( n \geq \kappa(G) \) there always exists a nontrivial homomorphism from \( G \) into~\( Sym_n \).

Now, let \( \mu(G) \) be the degree of a minimal faithful permutation representation of \( G \).
Obviously \( \kappa(G) \leq \mu(G) \) and the equality should not hold in general.
The following not widely known theorem of Berkovich tells us exactly when it holds.

\begin{theorem}[{\cite[Theorem~1]{berk}}]\label{core}
	Let \( G \) be a finite group. \( G \) is simple if and only if \( \kappa(G) = \mu(G) \).
\end{theorem}

As a consequence, if \( H \) is a proper subgroup of minimal index in \( G \), then \( G/core_G(H) \)
is a simple group, where \( core_G(H) = \bigcap_{g \in G} H^g \). This observation allows one to
search for subgroups of minimal index only in simple quotients of~\( G \). We have the following result.

\begin{theorem}\label{algoP}
	Let \( G \) be a finite permutation group given by generators. Then \( \kappa(G) \) can
	be computed in polynomial time in the degree of \( G \).
\end{theorem}

\begin{corollary}
	The group representability on trees where the group is presented
	as a permutation group via a generating set can be solved in polynomial time.
\end{corollary}

We note that in \cite{dutta} authors are mainly focused on groups given by Cayley tables,
so we in fact answered a more general question.

Notice that we do not claim to find the subgroup of minimal index itself (which is required to reconstruct
the corresponding action of a group on a tree). Nevertheless, in the case when
the group is given by its Cayley table, it is possible to enumerate all such subgroups.

\begin{theorem}\label{algoC}
	Let \( G \) be a finite group given by its Cayley table. Then the set \( \{ H < G \mid |G : H| = \kappa(G) \} \)
	can be computed in time polynomial in \( |G| \).
\end{theorem}

It might be very plausible that (at least one) subgroup of minimal index can be computed in polynomial time
in the case of permutation groups, but it most certainly would need a more advanced machinery.

The author would like to express his gratitude to prof.~Avinoam Mann, who pointed out that Theorem~\ref{core}
was proved earlier and gave the reference to Berkovich's paper.

\section{Proof of Theorem~\ref{core}}

The article \cite{berk} besides the original proof by Berkovich (originating in \cite{berkorig})
contains another very short and elegant proof attributed by the author to M.I.~Isaacs.
We reproduce it with almost no changes for the sake of completeness.

If \( G \) is simple, then clearly \( \kappa(G) = \mu(G) \). Therefore it suffices to prove the converse statement.

Let \( H \) be a subgroup of index \( \kappa(G) \) in \( G \) such that \( core_G(H) = 1 \).
Suppose that \( N \) is a nontrivial proper normal subgroup of \( G \). Since \( H \) is maximal,
we have \( G = NH \). Let \( U \) be a subgroup of \( H \) minimal with \( G = NU \). 
Obviously \( U > 1 \), and \( U \) does not lie in \( H^g \) for some \( g \in G \).
Set \( V = U \cap H^g < U \). We have
\[ |G : NV| = |NU : NV| = \frac{|N||U||N \cap V|}{|N||V||N \cap U|} \leq |U : V| < |G : H|,\]
since \( |U : V| = |UH^g : H^g| = |UH^g|/|H| \) and \( UH^g \subseteq HH^g \subset G \). 
By minimality of \( |G : H| \) it follows that \( G = NV \), contrary to the choice of \( U \).

\section{Proof of Theorem~\ref{algoP}}

In what follows, we assume the standard polynomial-time toolbox from \cite{seress}.

Let \( S \) be a simple group. Denote by \( O^S(G) \) the minimal
normal subgroup of~\( G \) such that each composition factor of \( G/O^S(G) \) is isomorphic to \( S \).
It is noted in \cite{seress} that an algorithm for computing \( O^S(G) \) in polynomial time is implicit in \cite{babai}.

Now let \( G \) be a permutation group given by its generators. Compute the composition series of \( G \),
and let \( \Sigma \) be the collection of isomorphism types of composition factors. By Theorem~\ref{core},
if \( H \) is a subgroup of minimal index, then it contains the maximal normal subgroup \( N = core_G(H) \). 
The quotient \( G/N \) is simple, therefore its isomorphism type \( S \) lies in \( \Sigma \) and \( O^S(G) \leq N < G \). 
Moreover, \( \kappa(G) = \kappa(G/N) = \mu(S) \), so
\[ \kappa(G) = \min \{ \mu(S) \mid S \in \Sigma, \, O^S(G) < G \}, \]
where \( \mu(S) \) can be found by checking the description of minimal faithful 
permutation representations of finite simple groups, which is well-known (for example, see \cite[Table~4]{guest}
for groups of Lie type and \cite[Table~4]{mazurov} for sporadic simple groups).
Since all steps can be performed in polynomial time, we obtain the required algorithm.

\section{Proof of Theorem~\ref{algoC}}

The key observation is the following.

\begin{lemma}\label{simpmax}
	Let \( G \) be a finite simple group given by its Cayley table.
	Then the set of maximal subgroups of \( G \) can be computed in time polynomial in \( |G| \).
\end{lemma}
\begin{proof}
	Try all possible \( 4 \)-tuples of elements of \( G \) (there are \( |G|^4 \) of those) 
	and generate corresponding subgroups. One can test in polynomial time 
	if a given subgroup is maximal, so we obtain the list of all maximal subgroups of \( G \)
	generated by \( 4 \) elements. By \cite[Theorem~1]{burness} every maximal subgroup of
	a finite simple group is \( 4 \)-generated, so we in fact found all maximal subgroups of \( G \).
\end{proof}

Set \( \mathcal{M}(G) = \{ N < G \mid N \textit{ is a normal subgroup of } G, \textit{ and } G/N \textit{ is simple} \} \),
and recall that we can compute \( \mathcal{M}(G) \) in polynomial time even for permutation groups
(see the proof of \cite[Lemma~7.4]{babai}). Notice that we can find the following set in polynomial time:
\[ \mathcal{A}_N = \{ H < G \mid N \leq H, \, |G : H| = \kappa(G) \}. \] 
Indeed, \( \kappa(G) \) can be computed in polynomial time by Theorem~\ref{algoP}, 
and obviously the Cayley table for \( G/N \) can be found in polynomial time, thus by Lemma~\ref{simpmax}
we can find all maximal subgroups of \( G/N \). By taking preimages and keeping only subgroups of index
equal to \( \kappa(G) \), we find the required set.

Now, by Theorem~\ref{core} every subgroup \( H \) with \( |G : H| = \kappa(G) \) contains a maximal normal subgroup.
Therefore \( \{ H < G \mid |G : H| = \kappa(G) \} = \bigcup_{N \in \mathcal{M}(G)} \mathcal{A}_N \), and this set
can be computed in polynomial time.

\noindent{\sl Saveliy V. Skresanov\\
Novosibirsk State University, 2 Pirogova Str.,\\
Novosibirsk, 630090, Russia\\
e-mail: s.skresanov@g.nsu.ru

\end{document}